\documentclass[10pt, reqno]{amsart}

\def\BibTeX{{\rm B\kern-.05em{\sc i\kern-.025em b}\kern-.08em
    T\kern-.1667em\lower.7ex\hbox{E}\kern-.125emX}}

\newtheorem{thm}{Theorem}[section]

\newtheorem{lem}[thm]{Lemma}

\newtheorem{cond}[thm]{Condition}

\theoremstyle{definition}

\theoremstyle{remark}
\newtheorem{rem}{Remark}[section]

\numberwithin{equation}{section}

    \newcommand{\floor}[1]{\lfloor#1\rfloor}

    \newcommand{\EE}{\mathbb{E}}

    \newcommand{\Exp}{\operatorname{E}}
    \newcommand{\E}{\Exp}

    \renewcommand{\Pr}{\operatorname{P}}

    \newcommand{\dto}{\xrightarrow{d}}
    \newcommand{\wto}{\xrightarrow{w}}
    \newcommand{\vto}{\xrightarrow{v}}
    \newcommand{\fidi}{\xrightarrow{\text{fidi}}}

    \newcommand{\rmd}{\mathrm{d}}

\newcommand{\be}{\begin{equation}}
    \newcommand{\ee}{\end{equation}}

\begin{document}

\title[Weak convergence of self-normalized partial sum processes] 
{A functional limit theorem for self-normalized partial sum processes in the $M_{1}$ topology}

%

\author{Danijel Krizmani\'{c}}

\address{Danijel Krizmani\'{c}\\ Faculty of Mathematics\\
        University of Rijeka\\
        Radmile Matej\v{c}i\'{c} 2, 51000 Rijeka\\
        Croatia}
\email{dkrizmanic@math.uniri.hr}





\begin{abstract}
For a stationary sequence of random variables we derive a self-normalized functional limit theorem under joint regular variation with index $\alpha \in (0,2)$ and weak dependence conditions. The convergence takes place in the space of real-valued c\`{a}dl\`{a}g functions on
$[0,1]$ with the Skorokhod $M_{1}$ topology.
\end{abstract}

\maketitle

\section{Introduction}
\label{intro}

For a stationary sequence of random variables $(X_{n})$ we study self-normalized partial sum processes
$\{ \zeta_{n}^{-1}\sum_{i=1}^{\floor {nt}}X_{i},\,t \in [0,1] \}$
as random elements of the space $D([0,1], \mathbb{R})$ of real valued c\`adl\`ag functions on $[0, 1]$,
where $\zeta_{n}^{2}=X_{1}^{2} + \ldots + X_{n}^{2}$. These processes arise naturally in the study of stochastic integrals, martingale inequalities and various statistical studies. Central limit theorems for partial sums $X_{1}+ \ldots + X_{n}$, $n \in \mathbb{N}$, and corresponding functional limit theorems are often derived using self-normalization techniques. Functional limit theorems (or invariance principles) play an important role in statistics and econometrics, and they have been studied by many authors in the literature. Among them we mention only Cs\"{o}rg\H{o} et al.~\cite{CSW03} for the case of independent random variables, Balan and Kulik~\cite{BaKu09} for $\phi$--mixing sequences, and Ra\v{c}kauskas and Suquet~\cite{RaSu11} for linear processes. For a survey on self-normalized limit theorems we refer to Shao and Wang~\cite{ShWa13}.

In this paper, we derive a functional limit theorem for the self-normalized partial sum process of stationary, regularly varying sequence of dependent random variables $(X_{n})$ with index $\alpha \in (0,2)$ for which clusters of high-treshold excesses can be broken down into asymptotically independent blocks:
\begin{equation}\label{e:convintro}
  \frac{1}{\zeta_{n}}  \sum_{i=1}^{\floor {n\,\cdot}}(X_{i}-c_{n}) \dto \frac{V(\,\cdot\,)}{W(1)} \qquad \textrm{as} \ n \to \infty,
\end{equation}
in $D([0,1], \mathbb{R})$ equipped with the Skorokhod $M_{1}$ topology under the condition
that all extremes within one such cluster have the same sign. Here $c_{n}$ are appropriate centering constants, $V$ and $W$ are stable L\'{e}vy process, and $\floor{x}$ represents the greatest integer not larger than $x$. In order to prove relation (\ref{e:convintro}) we will first establish functional convergence of the joint stochastic process
$$ L_{n}(t) = \bigg( \frac{1}{a_{n}}  \sum_{i=1}^{\floor {nt}}(X_{i}-c_{n}), \frac{1}{a_{n}^{2}}  \sum_{i=1}^{\floor {nt}}X_{i}^{2} \bigg), \qquad t \in [0,1]$$
in the space of $\mathbb{R}^{2}$--valued c\`{a}dl\`{a}g functions on
$[0,1]$ with the weak Skorokhod $M_{1}$ topology, and then we will apply the continuous mapping theorem to obtain (\ref{e:convintro}).

The paper is organized as follows. Section~\ref{S:Prel} is devoted to basic notions and results on regular variation, point processes and Skorokhod $M_{1}$ topology. In Section~\ref{S:jointconv2} we establish functional convergence of the joint stochastic process $L_{n}(\,\cdot\,)$,
and then in Section~\ref{S:Selfnormalized} we derive functional convergence of the self-normalized partial sum process.

\section{Preliminaries}\label{S:Prel}

\subsection{Regular variation}

Let $\EE^{d}=[-\infty, \infty]^{d} \setminus \{ 0 \}$. We equip
$\EE^{d}$ with the topology in which a set $B \subset \EE^{d}$
has compact closure if and only if it is bounded away from zero,
that is, if there exists $u > 0$ such that $B \subset \EE^{d}_u = \{ x
\in \EE^{d} : \|x\| >u \}$. Here $\| \cdot \|$ denotes the max-norm on $\mathbb{R}^{d}$, i.e.\
$\displaystyle \| x \|=\max \{ |x_{i}| : i=1, \ldots , d\}$ where
$x=(x_{1}, \ldots, x_{d}) \in \mathbb{R}^{d}$. Denote by $C_{K}^{+}(\EE^{d})$ the class of all
nonnegative, continuous functions on $\EE^{d}$ with compact support.

We say that a stationary process $(X_{n})_{n \in \mathbb{Z}}$ is \emph{(jointly) regularly varying} with index
$\alpha \in (0,\infty)$ if for any nonnegative integer $k$ the
$kd$-dimensional random vector $X = (X_{1}, \ldots , X_{k})$ is
multivariate regularly varying with index $\alpha$, i.e.\ there
exists a random vector $\Theta$ on the unit sphere
$\mathbb{S}^{kd-1} = \{ x \in \mathbb{R}^{kd} : \|x\|=1 \}$ such
that for every $u \in (0,\infty)$ and as $x \to \infty$,
 \begin{equation}\label{e:regvar1}
   \frac{\Pr(\|X\| > ux,\,X / \| X \| \in \cdot \, )}{\Pr(\| X \| >x)}
    \wto u^{-\alpha} \Pr( \Theta \in \cdot \,),
 \end{equation}
the arrow ''$\wto$'' denoting weak convergence of finite measures.
Univariate regular variation can be expressed in terms of vague convergence of
measures on $\EE = \EE^{1}$:
\begin{equation}
  \label{e:onedimregvar}
  n \Pr( a_n^{-1} X_i \in \cdot \, ) \vto \mu( \, \cdot \,),
\end{equation}
where $(a_{n})$ is a sequence of positive real numbers such that
\begin{equation}\label{e:niz}
 n \Pr ( |X_{1}| > a_{n}) \to 1
\end{equation}
as $n \to \infty$, and $\mu$ is a nonzero Radon measure on $\EE$ given by
\begin{equation}\label{e:mu}
  \mu(\rmd x) = \bigl( p \, 1_{(0, \infty)}(x) + q \, 1_{(-\infty, 0)}(x) \bigr) \, \alpha |x|^{-\alpha-1}\,\rmd x,
\end{equation}
for some $p \in [0,1]$, with $q=1-p$.

Theorem 2.1 in Basrak and Segers~\cite{BaSe} provides a convenient
characterization of joint regular variation:~it is necessary and
sufficient that there exists a process $(Y_n)_{n \in \mathbb{Z}}$
with $\Pr(|Y_0| > y) = y^{-\alpha}$ for $y \geq 1$ such that, as $x
\to \infty$,
\begin{equation}\label{e:tailprocess}
  \bigl( (x^{-1}\ X_n)_{n \in \mathbb{Z}} \, \big| \, | X_0| > x \bigr)
  \fidi (Y_n)_{n \in \mathbb{Z}},
\end{equation}
where "$\fidi$" denotes convergence of finite-dimensional
distributions. The process $(Y_{n})$ is called
the \emph{tail process} of $(X_{n})$.

\subsection{Point processes and dependence conditions}\label{s:pp}

Let $(X_{n})$ be a stationary sequence of random variables and assume it is jointly regularly varying with index $\alpha >0$. Let $(Y_{n})$ be
the tail process of $(X_{n})$. Let
\begin{equation*}\label{E:ppspacetime}
 N_{n} = \sum_{i=1}^{n} \delta_{(i / n,\,X_{i} / a_{n})} \qquad \textrm{for all} \ n\in \mathbb{N},
\end{equation*}
with $a_{n}$ as in (\ref{e:niz}). The point process convergence for the sequence $(N_{n})$ was already established by Basrak et al.~\cite{BKS} on the space $[0,1] \times \EE_{u}$ for any threshold $u>0$, with the limit depending on that threshold. A more useful convergence result for $N_{n}$ without the restrictions to various domains was obtained by Basrak and Tafro~\cite{BaTa16}. The appropriate weak dependence conditions for this convergence result are given below.

\begin{cond}\label{c:mixcond1}
There exists a sequence of positive integers $(r_{n})$ such that $r_{n} \to \infty $ and $r_{n} / n \to 0$ as $n \to \infty$ and such that for every $f \in C_{K}^{+}([0,1] \times \mathbb{E})$, denoting $k_{n} = \lfloor n / r_{n} \rfloor$, as $n \to \infty$,
\begin{equation}\label{e:mixcon}
 \E \biggl[ \exp \biggl\{ - \sum_{i=1}^{n} f \biggl(\frac{i}{n}, \frac{X_{i}}{a_{n}}
 \biggr) \biggr\} \biggr]
 - \prod_{k=1}^{k_{n}} \E \biggl[ \exp \biggl\{ - \sum_{i=1}^{r_{n}} f \biggl(\frac{kr_{n}}{n}, \frac{X_{i}}{a_{n}} \biggr) \biggr\} \biggr] \to 0.
\end{equation}
\end{cond}

\begin{cond}\label{c:mixcond2}
There exists a sequence of positive integers $(r_{n})$ such that $r_{n} \to \infty $ and $r_{n} / n \to 0$ as $n \to \infty$ and such that for every $u > 0$,
\begin{equation}
\label{e:anticluster}
  \lim_{m \to \infty} \limsup_{n \to \infty}
  \Pr \biggl( \max_{m \leq |i| \leq r_{n}} | X_{i} | > ua_{n}\,\bigg|\,| X_{0}|>ua_{n} \biggr) = 0.
\end{equation}
\end{cond}
Condition~\ref{c:mixcond1} is implied by the strong mixing property (see Krizmani\'{c}~\cite{Kr16}), with the sequence $(\xi_{n})$ being strongly mixing if $\alpha (n) \to 0$ as $n \to \infty$, where
$$\alpha (n) = \sup \{|\Pr (A \cap B) - \Pr(A) \Pr(B)| : A \in \mathcal{F}_{-\infty}^{0}, B \in \mathcal{F}_{n}^{\infty} \}$$
and $\mathcal{F}_{k}^{l} = \sigma( \{ \xi_{i} : k \leq i \leq l \} )$ for $-\infty \leq k \leq l \leq \infty$.
There are many time series satisfying these conditions, including moving averages, stochastic volatility and GARCH models (see for example Section 4 in Basrak et al.~\cite{BKS}).
By Proposition 4.2 in Basrak and Segers~\cite{BaSe},
under Condition~\ref{c:mixcond2} the following
holds
\begin{eqnarray}\label{E:theta:spectral}
   \theta := \Pr ({\textstyle\sup_{i\ge 1}} | Y_{i}| \le 1) = \Pr ({\textstyle\sup_{i\le -1}} | Y_{i}| \le 1)>0,
\end{eqnarray}
and $\theta$ is the extremal index of the univariate sequence $(| X_{n} |)$.
For a detailed discussion on joint regular variation and dependence Conditions~\ref{c:mixcond1} and \ref{c:mixcond2} we refer to Section 3.4 in Basrak et al.~\cite{BKS}.

Under joint regular variation and Conditions~\ref{c:mixcond1} and \ref{c:mixcond2}, by Theorem 3.1 in Basrak and Tafro~\cite{BaTa16}, as $n \to \infty$,
\begin{equation}\label{e:BaTa}
N_{n} \dto N = \sum_{i}\sum_{j}\delta_{(T_{i}, P_{i}\eta_{ij})}
\end{equation}
in $[0,1] \times \EE$, where $\sum_{i=1}^{\infty}\delta_{(T_{i}, P_{i})}$ is a Poisson process on $[0,1] \times (0,\infty)$
with intensity measure $Leb \times \nu$ where $\nu(\rmd x) = \theta \alpha
x^{-\alpha-1}1_{(0,\infty)}(x)\,\rmd x$, and $(\sum_{j= 1}^{\infty}\delta_{\eta_{ij}})_{i}$ is an i.i.d.~sequence of point processes in $\EE$ independent of $\sum_{i}\delta_{(T_{i}, P_{i})}$ and with common distribution equal to the distribution of $\sum_{j}\delta_{Z_{j}/L_{Z}}$, where $L_{Z}= \sup_{j \in \mathbb{Z}}|Z_{j}|$ and $\sum_{j}\delta_{Z_{j}}$ is distributed as $( \sum_{j \in \mathbb{Z}} \delta_{Y_j} \,|\, \sup_{i \le -1} | Y_i| \le 1).$

\subsection{The weak and strong $M_{1}$ topologies}\label{ss:j1m1}

 We start with the definition of the Skorokhod weak $M_{1}$ topology in a general space $D([0,1], \mathbb{R}^{d})$ of $\mathbb{R}^{d}$--valued c\`{a}dl\`{a}g functions on $[0,1]$. For $x \in D([0,1],
\mathbb{R}^{d})$ the completed (thick) graph of $x$ is the set
\[
  G_{x}
  = \{ (t,z) \in [0,1] \times \mathbb{R}^{d} : z \in [[x(t-), x(t)]]\},
\]
where $x(t-)$ is the left limit of $x$ at $t$ and $[[a,b]]$ is the product segment, i.e.
$[[a,b]]=[a_{1},b_{1}] \times [a_{2},b_{2}] \ldots \times [a_{d},b_{d}]$
for $a=(a_{1}, a_{2}, \ldots, a_{d}), b=(b_{1}, b_{2}, \ldots, b_{d}) \in
\mathbb{R}^{d}$.
 We define an
order on the graph $G_{x}$ by saying that $(t_{1},z_{1}) \le
(t_{2},z_{2})$ if either (i) $t_{1} < t_{2}$ or (ii) $t_{1} = t_{2}$
and $|x_{j}(t_{1}-) - z_{1j}| \le |x_{j}(t_{2}-) - z_{2j}|$
for all $j=1, 2, \ldots, d$. The relation $\le$ induces only a partial
order on the graph $G_{x}$. A weak parametric representation
of the graph $G_{x}$ is a continuous nondecreasing function $(r,u)$
mapping $[0,1]$ into $G_{x}$, with $r \in C([0,1],[0,1])$ being the
time component and $u \in C([0,1],
\mathbb{R}^{d})$ being the spatial component, such that $r(0)=0,
r(1)=1$ and $u(1)=x(1)$. Let $\Pi_{w}(x)$ denote the set of weak
parametric representations of the graph $G_{x}$. For $x_{1},x_{2}
\in D([0,1], \mathbb{R}^{d})$ define
\[
  d_{w}(x_{1},x_{2})
  = \inf \{ \|r_{1}-r_{2}\|_{[0,1]} \vee \|u_{1}-u_{2}\|_{[0,1]} : (r_{i},u_{i}) \in \Pi_{w}(x_{i}), i=1,2 \},
\]
where $\|x\|_{[0,1]} = \sup \{ \|x(t)\| : t \in [0,1] \}$. Now we
say that $x_{n} \to x$ in $D([0,1], \mathbb{R}^{d})$ for a sequence
$(x_{n})$ in the weak Skorokhod $M_{1}$ (or shortly $WM_{1}$)
topology if $d_{w}(x_{n},x)\to 0$ as $n \to \infty$.

Now we define the strong $M_{1}$ topology. For $x \in D([0,1], \mathbb{R}^{d})$
let
\[
  \Gamma_{x}
  = \{ (t,z) \in [0,1] \times \mathbb{R}^{d} : z \in [x(t-), x(t)] \},
\]
where $[a,b] = \{  \lambda a + (1-\lambda)b : 0 \leq \lambda \leq 1 \}$ for $a, b \in \mathbb{R}^{d}$. We say $(r,u)$ is a strong parametric representation of $\Gamma_{x}$ if it is a continuous nondecreasing function mapping $[0,1]$ onto $\Gamma_{x}$. Denote by $\Pi(x)$ the set of all strong parametric representations of the completed (thin) graph $\Gamma_{x}$. Then for $x_{1},x_{2} \in D([0,1], \mathbb{R}^{d})$ put
\[
  d_{M_{1}}(x_{1},x_{2})
  = \inf \{ \|r_{1}-r_{2}\|_{[0,1]} \vee \|u_{1}-u_{2}\|_{[0,1]} : (r_{i},u_{i}) \in \Pi(x_{i}), i=1,2 \}.
\]
The distance function $d_{M_{1}}$ is a metric on $D([0,1], \mathbb{R}^{d})$, and the induced topology is called the (standard or strong) Skorokhod $M_{1}$ topology.
The $WM_{1}$
topology is weaker than the standard $M_{1}$ topology on $D([0,1],
\mathbb{R}^{d})$, but they coincide if $d=1$. The $WM_{1}$ topology
coincides with the topology induced by the metric
\begin{equation}\label{e:defdp}
 d_{p}(x_{1},x_{2})= \max \{ d_{M_{1}}(x_{1j},x_{2j}) : j=1,2, \ldots, d \}
\end{equation}
 for $x_{i}=(x_{i1}, x_{i2}, \ldots, x_{id}) \in D([0,1],
 \mathbb{R}^{d})$ and $i=1,2$. The metric $d_{p}$ induces the product topology on $D([0,1], \mathbb{R}^{d})$.
For detailed discussion of the strong and weak $M_{1}$ topologies we refer to sections 12.3--12.5 in Whitt~\cite{Whitt02}.

Denote by $C^{\uparrow}_{0}([0,1], \mathbb{R})$ the subset of functions $x$ in $D[0,1], \mathbb{R})$ that are continuous and nondecreasing with $x(0)>0$. It is a measurable subset of $D([0,1], \mathbb{R})$ with the Skorokhod $M_{1}$ topology, see Krizmani\'{c}~\cite{Kr23}.
The following auxiliary result which will be used in the last section deals with $M_{1}$ continuity of division of two c\`{a}dl\`{a}g functions.

\begin{lem}[Krizmani\'{c}~\cite{Kr23}]\label{l:M1div}
The function $h \colon D([0,1], \mathbb{R}) \times C^{\uparrow}_{0}([0,1], \mathbb{R}) \to D([0,1], \mathbb{R})$ defined by
$h(x,y)= x/y$
is continuous
when $D([0,1], \mathbb{R}) \times C^{\uparrow}_{0}([0,1], \mathbb{R})$ is endowed with the weak $M_{1}$ topology and $D([0,1], \mathbb{R})$ is endowed with the standard $M_{1}$ topology.
\end{lem}

For $a \in \mathbb{R}$ denote by $\widehat{a}$ the constant function in $D([0,1], \mathbb{R})$ defined by $\widehat{a}(t)=a$ for all $t \in [0,1]$. Let $D_{m}([0,1], \mathbb{R})$ be the subset of functions $x$ in $D([0,1], \mathbb{R})$ that are monotone with $x(0) \geq 0$. Similar to Corollary 1 in Krizmani\'{c}~\cite{Kr23} one can prove the following result.

\begin{lem}\label{l:M1contconst}
The function $f \colon D([0,1], \mathbb{R}) \times D_{m}([0,1] \times \mathbb{R}) \to D([0,1], \mathbb{R}) \times D_{m}([0,1], \mathbb{R}) $ defined by
$ f(x,y) = (x, \widehat{y(1)})$
is continuous when  $D([0,1], \mathbb{R}) \times D_{m}([0,1], \mathbb{R})$ is endowed with the weak $M_{1}$ topology.
\end{lem}

\section{Joint functional convergence of partial sum processes with $\alpha \in (0,2)$}
\label{S:jointconv2}

Let $(X_{n})$ be a stationary sequence of random variables, jointly regularly varying with index $\alpha \in (0,2)$, and assume Condition \ref{c:mixcond2} hold. In the proof of the following theorem we will use an auxiliary result from Basrak et al.~\cite{BaPlSo} which holds under $\beta$--mixing, and hence we assume the sequence $(X_{n})$ to be $\beta$--mixing. Since $\beta$--mixing implies strong mixing, Condition~\ref{c:mixcond1} automatically holds. For $n \in \mathbb{N}$ let
$$ L_{n}(t) =  \bigg( \sum_{k=1}^{\lfloor nt \rfloor} \frac{X_{k}}{a_{n}} - \lfloor nt \rfloor b_{n}, \sum_{k=1}^{\lfloor nt \rfloor}\frac{X_{k}^{2}}{a_{n}^{2}} \bigg),
   \qquad t \in [0,1],$$
with $a_{n}$ as in (\ref{e:niz}) and
$$b_{n} = \mathrm{E} \bigg( \frac{X_{1}}{a_{n}} 1_{\big\{ \frac{|X_{1}|}{a_{n}} \leq 1 \big\}} \bigg).$$
The main idea is to represent $L_n$ as the image of the
time-space point process $N_n$ under an appropriate summation
functional. Then, using certain continuity properties of this
functional, by the continuous mapping theorem we
transfer the weak convergence of $N_n$ in (\ref{e:BaTa}) to
weak convergence of $L_n$.
In the case $\alpha \in [1,2)$ we need the following additional assumption to deal with small jumps.
\begin{cond}\label{c:step6cond}
For all $\delta > 0$,
$$
  \lim_{u \downarrow 0} \limsup_{n \to \infty} \Pr \bigg[
  \max_{0 \le k \le n}  \bigg| \sum_{i=1}^{k} \bigg( \frac{X_{i}}{a_{n}}
  1_{ \big\{ \frac{|X_{i}|}{a_{n}} \le u \big\} } -  \E \bigg( \frac{X_{i}}{a_{n}}
  1_{ \big\{ \frac{|X_{i}|}{a_{n}} \le u \big\} } \bigg) \bigg) \bigg| > \delta
  \bigg]=0.$$
\end{cond}
Condition~\ref{c:step6cond} can be hard to check for dependent sequences, but for instance it holds for $\rho$--mixing processes with a certain rate (see Jakubowski and Kobus~\cite{JaKo89} and Tyran-Kami\'{n}ska~\cite{TK2010}).

Denote by $\sum_{j}\delta_{\eta_{j}}$ a point process with the distribution equal to the distribution of $\sum_{j}\delta_{\eta_{1j}}$, where the latter point process is described in (\ref{e:BaTa}).
For $ \alpha \leq 1$ it holds that
\begin{equation}\label{e:MW1}
 \mathrm{E} \Big( \sum_{j}|\eta_{j}| \Big)^{\alpha} < \infty
\end{equation}
(see Davis and Hsing~\cite{DaHs95}), but it may fail for $\alpha > 1$ (see Mikosch and Wintenberger~\cite{MiWi14}).
Hence in the case $\alpha \in [1,2)$ we will assume additionally the following condition.
\begin{cond}\label{c:BPScond1}
It holds that
\begin{equation}\label{e:BPScond1}
  \left. \begin{array}{lc}
                                   \mathrm{E} \bigg[ \sum_{j}\eta_{j} \log \Big( |\eta_{j}|^{-1} \sum_{i}|\eta_{i}| \Big) \bigg] < \infty, & \quad \alpha =1,\\[1em]
                                   \mathrm{E} \Big( \sum_{j}|\eta_{j}| \Big)^{\alpha} < \infty, & \quad \alpha \in (1,2),
                                 \end{array}\right.
\end{equation}
with the convention $\eta_{j} \log ( |\eta_{j}|^{-1} \sum_{i}|\eta_{i}|)=0$ if $\eta_{j}=0$.
\end{cond}
For a discussion about the condition for $\alpha=1$ in (\ref{e:BPScond1}) see Remark 4.8 in Basrak et al.~\cite{BaPlSo}.
Fix $0 < u < \infty$ and define the summation functional
$$ \Phi^{(u)} \colon \mathbf{M}_{p}([0,1] \times \EE) \to D([0,1], \mathbb{R}^{2})$$
by
$$ \Phi^{(u)} \Big( \sum_{i}\delta_{(t_{i}, x_{i})} \Big) (t)
  =  \Big( \sum_{t_{i} \leq t}x_{i}\,1_{\{u < |x_{i}| < \infty \}},  \sum_{t_{i} \leq t} x_{i}^{2}\,1_{\{u < |x_{i}| < \infty \}}  \Big), \qquad t \in [0,1],$$
  where the space $\mathbf{M}_p([0,1] \times \EE)$ of Radon point
measures on $[0,1] \times \EE$ is equipped with the vague
topology (see Resnick~\cite{Re87}, Chapter 3). Let $\Lambda = \Lambda_{1} \cap \Lambda_{2}$, where
\begin{multline*}
 \Lambda_{1} =
 \{ \eta \in \mathbf{M}_{p}([0,1] \times \EE) :
   \eta ( \{0,1 \} \times \EE) = 0 = \eta ([0,1] \times \{ \pm \infty, \pm u \}) \}, \\[1em]
 \shoveleft \Lambda_{2} =
 \{ \eta \in \mathbf{M}_{p}([0,1] \times \EE) :
  \eta ( \{ t \} \times (u, \infty]) \cdot \eta ( \{ t \} \times [-\infty,-u)) = 0 \\
  \text{for all $t \in [0,1]$} \}.
\end{multline*}
The elements
of $\Lambda_2$ have the property that atoms in $[0,1] \times (\EE \setminus [-u,u])$ with the same time
coordinate are all on the same side of the time axis. Similar to Lemma 3.1 in Basrak et al.~\cite{BKS} one can prove that the Poisson cluster process $N$ in (\ref{e:BaTa}) a.s.~belongs to the set $\Lambda$ given that each cluster of large values of $(X_{n})$ contains only values of the same sign. Precisely,

\begin{lem}\label{l:prob1}
Assume that with probability one the tail process $(Y_{i})_{i \in \mathbb{Z}}$ of the sequence $(X_{n})$ has no two values of the opposite sign. Then
$ \Pr ( N \in \Lambda ) = 1$.
\end{lem}

In the theorem below we establish functional convergence of the process $L_{n}$, with the limit $L = (V, W)$ consisting of stable L\'{e}vy processes $V$ and $W$. The distribution of a L\'{e}vy process $V$ is characterized by its
characteristic triple, that is, the characteristic triple of the infinitely divisible distribution
of $V(1)$. The characteristic function of $V(1)$ and the characteristic triple
$(a, \nu, b)$ are related in the following way:
\begin{equation}\label{e:Kintchin}
  \mathrm{E} [e^{izV(1)}] = \exp \biggl( -\frac{1}{2}az^{2} + ibz + \int_{\mathbb{R}} \bigl( e^{izx}-1-izx 1_{[-1,1]}(x) \bigr)\,\nu(\rmd x) \biggr)
\end{equation}
for $z \in \mathbb{R}$. Here $a \ge 0$, $b \in \mathbb{R}$ are constants, and $\nu$ is a measure on $\mathbb{R}$ satisfying
$$ \nu ( \{0\})=0 \qquad \text{and} \qquad \int_{\mathbb{R}}(|x|^{2} \wedge 1)\,\nu(\rmd x) < \infty.$$

\begin{thm}\label{t:functconvergence2}
Let $(X_{n})$ be a stationary sequence of random variables, which is $\beta$--mixing and jointly regularly varying with index $\alpha \in (0,2)$, and of which the tail process $(Y_{i})_{i \in \mathbb{Z}}$ almost surely has no two values of the opposite sign. Suppose that Condition~\ref{c:mixcond2} holds. If $\alpha \in [1,2)$ also suppose that Conditions~\ref{c:step6cond} and \ref{c:BPScond1} hold.
Then
$L_{n} \dto L$ as $n \to \infty$,
in $D([0,1], \mathbb{R}^{2})$ endowed with the weak $M_{1}$ topology, where $L = (V,W)$, $V$
is an $\alpha$--stable L\'{e}vy process with characteristic triple $(0,
\nu_{1}, \gamma_{1})$ and $W$ is an $\alpha/2$--stable L\'{e}vy process with characteristic triple $(0,
\nu_{2}, \gamma_{2})$,
 where
\begin{align*}
 \nu_{1}(\rmd x) &= \theta \alpha \big( c_{+} 1_{(0,\infty)}(x) + c_{-} 1_{(-\infty, 0)}(x) \big) |x|^{-\alpha -1}\,\rmd x,\\
\nu_{2}(\rmd x) &=  \frac{\theta \alpha}{2} \mathrm{E}\bigg( \sum_{j}\eta_{j}^{2}\bigg)^{\alpha/2} x^{-\alpha/2 -1}1_{(0,\infty)}(x) \,\rmd x,\\
\gamma_{1} &= \left\{ \begin{array}{ll}
                                   \frac{\alpha}{\alpha-1} \big( p-q - \theta (c_{+}-c_{-}) \big), & \alpha \in (0,1) \cup (1,2),\\[0.2em]
                                   - \theta\mathrm{E}\bigg[ \sum_{j}\eta_{j} \log \Big( \Big| \sum_{i}\eta_{i}\eta_{j}^{-1} \Big|  \Big) \bigg], & \alpha =1,
                                 \end{array}\right.\\
\gamma_{2} &= \frac{\theta \alpha}{2-\alpha} \mathrm{E}\bigg( \sum_{j}\eta_{j}^{2}\bigg)^{\alpha/2},
\end{align*}
with $c_{+} = \mathrm{E} [ ( \sum_{j}\eta_{j} )^{\alpha} 1_{ \{ \sum_{j}\eta_{j} > 0 \}} ]$,
   $c_{-} = \mathrm{E} [( - \sum_{j}\eta_{j} )^{\alpha} 1_{ \{ \sum_{j}\eta_{j} < 0 \}}]$, and $p$ and $q$ as in $(\ref{e:mu})$.
\end{thm}

\begin{proof}
For an arbitrary $u>0$ the summation functional $\Phi^{(u)}$ is continuous on the set $\Lambda$ (see Lemma 4 in Krizmani\'{c}~\cite{Kr23}), and by
Lemma~\ref{l:prob1} this set almost surely contains the limiting point process $N$ from (\ref{e:BaTa}). Hence an application of the continuous mapping theorem yields
$\Phi^{(u)}(N_{n}) \dto \Phi^{(u)}(N)$ in $D([0,1], \mathbb{R}^{2})$ under the weak $M_{1}$ topology, that is
\begin{multline}\label{e:mainconvC2}
      L_{n}^{(u)}(\,\cdot\,) := \bigg( \sum_{i = 1}^{\lfloor n \, \cdot \, \rfloor} \frac{X_{i}}{a_{n}}
    1_{ \bigl\{ \frac{|X_{i}|}{a_{n}} > u \bigr\} },  \sum_{i = 1}^{\lfloor n \, \cdot \, \rfloor} \frac{X_{i}^{2}}{a_{n}^{2}}
    1_{ \bigl\{ \frac{|X_{i}|}{a_{n}} > u \bigr\} } \bigg) \\
    \dto L^{(u)}(\,\cdot\,) :=  \bigg( \sum_{T_{i} \le \, \cdot} \sum_{j}P_{i}\eta_{ij} 1_{\{ P_{i}|\eta_{ij}| > u \}}, \sum_{T_{i} \le \, \cdot} \sum_{j}P_{i}^{2}\eta_{ij}^{2} 1_{\{ P_{i}|\eta_{ij}| > u \}} \bigg).
\end{multline}
From (\ref{e:onedimregvar}) we have, as $n \to \infty$,
\begin{eqnarray}\label{e:conv12}
 \nonumber \floor{nt} \mathrm{E} \Big( \frac{X_{1}}{a_{n}} 1_{ \big\{ u < \frac{|X_{1}|}{a_{n}} \leq 1 \big\} } \Big) &=& \frac{\floor{nt}}{n} \int_{u < |x| \leq 1} x n\,\Pr \Big( \frac{X_{1}}{a_{n}} \in \rmd x \Big) \\[0.6em]
   & \to &  t \int_{u < |x| \leq 1} x\mu(\rmd x)
\end{eqnarray}
for every $t \in [0,1]$, and this convergence is uniform in $t$. Therefore applying Lemma 2.1 from Krizmani\'{c}~\cite{Kr18} (adjusted for the $M_{1}$ convergence) to (\ref{e:mainconvC2}) and (\ref{e:conv12}) we obtain, as $n \to \infty$,
\begin{multline}\label{e:mainconvC3}
      \widetilde{L}_{n}^{(u)}(\,\cdot\,) := \bigg( \sum_{i = 1}^{\lfloor n \, \cdot \, \rfloor} \frac{X_{i}}{a_{n}}
    1_{ \bigl\{ \frac{|X_{i}|}{a_{n}} > u \bigr\} } - \floor{n\,\cdot\,}b_{n}^{(u)}, \sum_{i = 1}^{\lfloor n \, \cdot \, \rfloor} \frac{X_{i}^{2}}{a_{n}^{2}}
    1_{ \bigl\{ \frac{|X_{i}|}{a_{n}} > u \bigr\} } \bigg) \\
    \dto \widetilde{L}^{(u)}(\,\cdot\,) :=  \bigg( \sum_{T_{i} \le \, \cdot} \sum_{j}P_{i}\eta_{ij} 1_{\{ P_{i}|\eta_{ij}| > u \}} - (\,\cdot\,) b^{(u)}, \sum_{T_{i} \le \, \cdot} \sum_{j}P_{i}^{2}\eta_{ij}^{2} 1_{\{ P_{i}|\eta_{ij}| > u \}} \bigg),
\end{multline}
where
$$b_{n}^{(u)} = \mathrm{E} \Big( \frac{X_{1}}{a_{n}} 1_{ \big\{ u < \frac{|X_{1}|}{a_{n}} \leq 1 \big\} } \Big) \qquad \textrm{and} \qquad
 b^{(u)} =  \int_{u < |x| \leq 1} x\mu(\rmd x).$$
Let $U_{i} = \sum_{j}\eta_{ij}^{2}$, $i=1,2,\ldots$, and note that
by relation (\ref{e:MW1}) if $\alpha \in (0,1]$ and Condition~\ref{c:BPScond1} if $\alpha \in (1,2)$ we have
\begin{equation*}\label{e:MW2-2}
 \mathrm{E}(U_{1})^{\alpha/2} = \mathrm{E} \Big( \sum_{j}|\eta_{j}|^{2} \Big)^{\alpha/2} \leq \mathrm{E} \Big( \sum_{j}|\eta_{j}| \Big)^{\alpha} < \infty,
\end{equation*}
where the first inequality in the relation above holds by the triangle inequality $|\sum_{i}a_{i}|^{s} \leq \sum_{i}|a_{i}|^{s}$ with $s \in (0,1]$.
Proposition 5.2 and Proposition 5.3 in Resnick~\cite{Resnick07} imply that $P_{i}^{2}U_{i}$, $i=1,2,\ldots$, are the points of a Poisson process with intensity measure $(\theta \alpha /2) \mathrm{E}(U_{1})^{\alpha/2} x^{-\alpha/2-1}\,\rmd x$ for $x>0$. Since $\alpha/2 < 1$, these points are summable (see the proof of Theorem 3.1 in Davis and Hsing~\cite{DaHs95}). It follows that for all $t\in [0,1]$
$$ \sum_{T_{i} \leq t} \sum_{j} P_{i}^{2}\eta_{ij}^{2} 1_{\{ P_{i}|\eta_{ij}| > u \}} \to \sum_{T_{i} \leq t} \sum_{j}P_{i}^{2}\eta_{ij}^{2}$$
almost surely as $u \to 0$, and hence by the dominated convergence theorem
$$ \sup_{t \in [0,1]} \bigg| \sum_{T_{i} \leq t} \sum_{j} P_{i}^{2}\eta_{ij}^{2} 1_{\{ P_{i}|\eta_{ij}| > u \}} - \sum_{T_{i} \leq t} \sum_{j}P_{i}^{2}\eta_{ij}^{2} \bigg| \leq \sum_{i=1}^{\infty}\sum_{j}P_{i}^{2}|\eta_{ij}|^{2} 1_{ \{ P_{i}|\eta_{ij}| \leq u \} } \to 0$$
almost surely as $u \to 0$.
Since uniform convergence implies Skorokhod $M_{1}$ convergence, we get
\begin{equation}\label{e:dm12N}
d_{M_{1}} \Big( \sum_{T_{i} \leq\,\cdot} \sum_{j} P_{i}^{2}\eta_{ij}^{2} 1_{\{ P_{i}|\eta_{ij}| > u \}}, \sum_{T_{i} \leq\,\cdot} \sum_{j}P_{i}^{2}\eta_{ij}^{2} \Big) \to 0
\end{equation}
almost surely as $u \to 0$. By Proposition 5.2 and Proposition 5.3 in Resnick~\cite{Resnick07}, the process
$$\sum_{i} \delta_{(T_{i}, \sum_{j}P_{i}^{2}\eta_{ij}^{2}})$$
 is a Poisson process with intensity measure
$ Leb \times \nu_{2}$.
By the It\^{o} representation of the L\'{e}vy process (see Resnick~\cite{Resnick07}, pp.~150--153) and Theorem 14.3 in Sato~\cite{Sa99},
$$\sum_{T_{i} \leq\,\cdot} \sum_{j}P_{i}^{2}\eta_{ij}^{2}$$
is an $\alpha/2$--stable L\'{e}vy process with characteristic triple $(0, \nu_{2}, \theta \alpha \mathrm{E}(\sum_{j}\eta_{j}^{2})^{\alpha/2}/(2-\alpha))$, i.e. $(0,\nu_{2}, \gamma_{2})$.

Now we treat separately the cases $\alpha \in (0,1)$ and $\alpha \in [1,2)$.
Assume first $\alpha \in (0,1)$. By denoting $\widetilde{U}_{i} = \sum_{j}\eta_{ij}$, $i=1,2,\ldots$,
we see that
\begin{equation*}\label{e:MWnew}
 \mathrm{E}|\widetilde{U}_{1}|^{\alpha} \leq \mathrm{E} \Big( \sum_{j}|\eta_{j}| \Big)^{\alpha} < \infty,
\end{equation*}
and by the same arguments as above we obtain
\begin{equation}\label{e:dm01N}
d_{M_{1}} \Big( \sum_{T_{i} \leq\,\cdot} \sum_{j} P_{i}\eta_{ij} 1_{\{ P_{i}|\eta_{ij}| > u \}}, \sum_{T_{i} \leq\,\cdot} \sum_{j}P_{i}\eta_{ij} \Big) \to 0
\end{equation}
almost surely as $u \to 0$. Using the same arguments as before for process $\sum_{i} \delta_{(T_{i}, \sum_{j}P_{i}^{2}\eta_{ij}^{2}})$ we conclude that
$$\sum_{i} \delta_{(T_{i}, \sum_{j}P_{i}\eta_{ij} })$$
 is a Poisson process with intensity measure
$ Leb \times \nu_{1}$, and that the process
$$ \sum_{T_{i} \leq\,\cdot} \sum_{j}P_{i}\eta_{ij}$$
 is an $\alpha$--stable L\'{e}vy process with characteristic triple $(0, \nu_{1}, (c_{+}-c_{-}) \theta \alpha / (1-\alpha))$. Since
$$ b^{(u)} = \int_{u < |x| \leq 1} x\mu(\rmd x) \to (p-q) \frac{\alpha}{1-\alpha} \qquad \textrm{as} \ u \to 0,$$
with $p$ and $q$ as in $(\ref{e:mu})$, we have, as $u \to 0$,
$$ (\,\cdot\,) b^{(u)} \to (\,\cdot\,) (p-q) \frac{\alpha}{1-\alpha}$$
in $D([0,1], \mathbb{R})$. The latter function is continuous, and hence from (\ref{e:dm01N}) by applying Corollary 12.7.1 in Whitt~\cite{Whitt02} we obtain
\begin{equation}\label{e:dm01N2}
d_{M_{1}} \Big( \sum_{T_{i} \leq\,\cdot} \sum_{j} P_{i}\eta_{ij} 1_{\{ P_{i}|\eta_{ij}| > u \}} - (\,\cdot\,)b^{(u)}, V_{\alpha}(\,\cdot\,) \Big) \to 0
\end{equation}
almost surely as $u \to 0$, where
$$ V_{\alpha}(t) = \sum_{T_{i} \leq t} \sum_{j}P_{i}\eta_{ij} - t (p-q)\frac{\alpha}{1-\alpha}, \qquad t \in [0,1],$$
is an $\alpha$--stable L\'{e}vy process with characteristic triple
$$\Big( 0, \nu_{1}, \frac{\alpha}{1-\alpha} (\theta (c_{+}-c_{-}) - (p-q)) \Big),$$
 i.e. $(0, \nu_{1}, \gamma_{1})$.

Assume now $\alpha \in [1,2)$.
Under the $\beta$--mixing property and Conditions~\ref{c:mixcond2}, \ref{c:step6cond} and \ref{c:BPScond1}, Lemma 6.4 in Basrak et al.~\cite{BaPlSo} implies the existence of an $\alpha$--stable L\'{e}vy process $V_{\alpha}$ such that, as $u \to 0$ (along some subsequence)
$$ \sum_{T_{i} \le \, \cdot} \sum_{j}P_{i}\eta_{ij} 1_{\{ P_{i}|\eta_{ij}| > u \}} - (\,\cdot\,) b^{(u)} \to V_{\alpha}(\,\cdot\,)$$
uniformly almost surely. Again, since uniform convergence implies Skorokhod $M_{1}$ convergence, it follows that
\begin{equation}\label{e:dm13N}
d_{M_{1}} \Big( \sum_{T_{i} \leq\,\cdot} \sum_{j} P_{i}\eta_{ij} 1_{\{ P_{i}|\eta_{ij}| > u \}} - (\,\cdot\,) b^{(u)}, V_{\alpha}(\,\cdot\,) \Big) \to 0
\end{equation}
almost surely as $u \to 0$. By the same arguments used in the proof of Theorem 3.4 in Krizmani\'{c}~\cite{Kr20} it can be shown that $(0,\nu_{1}, \gamma_{1})$ is the characteristic triple of the process $V_{\alpha}$, and hence we omit it here.

Recalling the definition of the metric $d_{p}$ in (\ref{e:defdp}), from (\ref{e:dm12N}), (\ref{e:dm01N2}) and (\ref{e:dm13N}) we obtain
\begin{equation*}
d_{p}(\widetilde{L}^{(u)}, L) \to 0
\end{equation*}
almost surely as $u \to 0$, where
$$L(t) = \bigg(V_{\alpha}(t) , \sum_{T_{i} \le t} \sum_{j}P_{i}^{2}\eta_{ij}^{2} \bigg),$$
with
$$ V_{\alpha}(t) = \left\{ \begin{array}{lc}
                                  \displaystyle \sum_{T_{i} \leq t} \sum_{j}P_{i}\eta_{ij} - t (p-q)\frac{\alpha}{1-\alpha}, & \quad \alpha \in (0,1),\\[1.2em]
                                  \displaystyle \lim_{u \to 0} \bigg( \sum_{T_{i} \le t} \sum_{j}P_{i}\eta_{ij} 1_{\{ P_{i}|\eta_{ij}| > u \}} - t b^{(u)} \bigg), & \quad \alpha \in [1,2).
                                 \end{array}\right.$$
 Since almost sure convergence implies weak convergence, we have, as $u \to 0$,
\begin{equation}\label{e:mainconvC4}
 \widetilde{L}^{(u)} \dto L
\end{equation}
in $D([0,1], \mathbb{R}^{2})$ endowed with the weak $M_{1}$ topology.

If we show that
$$ \lim_{u \to 0}\limsup_{n \to \infty} \Pr(d_{p}(L_{n},\widetilde{L}_{n}^{(u)}) > \epsilon)=0,$$
for every $\epsilon >0$,
from (\ref{e:mainconvC3}) and (\ref{e:mainconvC4})
by a variant of Slutsky's theorem (see Theorem 3.5 in Resnick~\cite{Resnick07}) it will follow that
$ L_{n} \dto L$ as $n \to \infty$,
in $D([0,1], \mathbb{R}^{2})$ with the weak $M_{1}$ topology.
Since the
 metric $d_{p}$ on $D([0,1], \mathbb{R}^{2})$ is bounded above by the uniform metric on
 $D([0,1], \mathbb{R}^{2})$ (see Theorem 12.10.3 in Whitt~\cite{Whitt02}), it suffices to show that
 $$ \lim_{u \to 0} \limsup_{n \to \infty} \Pr \biggl(
 \sup_{t \in [0,1]} \|L_{n}(t) - \widetilde{L}_{n}^{(u)}(t)\| >
 \epsilon \biggr)=0.$$
Recalling the definitions of $L_{n}$ and $\widetilde{L}_{n}^{(u)}$, we have
\begin{eqnarray}\label{e:slutsky-C2}
    \nonumber \Pr \bigg(
     \sup_{t \in [0,1]} \|L_{n}(t) - \widetilde{L}_{n}^{(u)}(t)\| >  \epsilon \bigg) & & \\[0.3em]
   \nonumber & \hspace*{-28.5em} = & \ \hspace*{-14.5em} \Pr \bigg(
       \sup_{t \in [0,1]} \ \max \bigg\{  \bigg| \sum_{i=1}^{\lfloor nt \rfloor} \frac{X_{i}}{a_{n}}
       1_{ \big\{ \frac{|X_{i}|}{a_{n}} \leq u \big\} } - \lfloor nt \rfloor  \mathrm{E} \Big( \frac{X_{1}}{a_{n}} 1_{ \big\{  \frac{|X_{1}|}{a_{n}} \leq u \big\} } \Big) \bigg|,\bigg| \sum_{i=1}^{\lfloor nt \rfloor} \frac{X_{i}^{2}}{a_{n}^{2}}
       1_{ \big\{ \frac{|X_{i}|}{a_{n}} \leq u \big\} } \bigg|  \bigg\}  > \epsilon
       \bigg)\\[0.4em]
   \nonumber & \hspace*{-28.5em} \leq & \ \hspace*{-14.5em} \Pr \bigg( \max_{0 \leq k \leq n}
       \bigg| \sum_{i=1}^{k} \bigg( \frac{X_{i}}{a_{n}}
       1_{ \big\{ \frac{|X_{i}|}{a_{n}} \leq u \big\} } - \mathrm{E} \bigg( \frac{X_{i}}{a_{n}}
       1_{ \big\{ \frac{|X_{i}|}{a_{n}} \leq u \big\} } \bigg) \bigg) \bigg| > \epsilon
       \bigg) + \Pr \bigg(
       \sum_{i=1}^{n} \frac{X_{i}^{2}}{a_{n}^{2}}
       1_{ \big\{ \frac{|X_{i}|}{a_{n}} \leq u \big\} }  > \epsilon
       \bigg).
 \end{eqnarray}
 For $\alpha \in [1,2)$ the first term on the right-hand side of the above relation by Condition~\ref{c:step6cond} tends to zero if we first let $n \to \infty$ and then $u \to 0$. The same conclusion is valid for $\alpha \in (0,1)$, since in this case Condition~\ref{c:step6cond} automatically holds (see the end of the proof of Theorem 4.1 in Tyran-Kami\'{n}ska~\cite{TK2010}).

 Using stationarity and Markov's inequality we bound the second term on the right-hand side of the last relation above by
 $$\Pr \bigg(
       \sum_{i=1}^{n} \frac{X_{i}^{2}}{a_{n}^{2}}
       1_{ \big\{ \frac{|X_{i}|}{a_{n}} \leq u \big\} }  > \epsilon
       \bigg) \leq \epsilon^{-1} n \mathrm{E} \bigg( \frac{X_{1}^{2}}{a_{n}^{2}}
       1_{ \big\{ \frac{|X_{1}|}{a_{n}} \leq u \big\} } \bigg).$$
 Note that
  $$n \mathrm{E} \bigg( \frac{X_{1}^{2}}{a_{n}^{2}}
       1_{ \big\{ \frac{|X_{1}|}{a_{n}} \leq u \big\} } \bigg) =  u^{2} \cdot n \Pr (|X_{1}| > a_{n}) \cdot \frac{\Pr(|X_{1}| > ua_{n})}{\Pr(|X_{1}|>a_{n})} \cdot
          \frac{\mathrm{E}(|X_{1}|^{2} 1_{ \{ |X_{1}| \leq ua_{n} \} })}{ u^{2}a_{n}^{2} \Pr (|X_{1}| >ua_{n})}.$$
 Since $X_{1}$ is a regularly varying random variable with index $\alpha$, it follows immediately that
  $$ \frac{\Pr(|X_{1}| > ua_{n})}{\Pr(|X_{1}|>a_{n})} \to u^{-\alpha}$$
  as $n \to \infty$. By Karamata's theorem
  $$ \lim_{n \to \infty} \frac{\mathrm{E}(|X_{1}|^{2} 1_{ \{ |X_{1}| \leq ua_{n} \} })}{ u^{2}a_{n}^{2} \Pr (|X_{1}| >ua_{n})} = \frac{\alpha}{2-\alpha}.$$
 Taking into account relation (\ref{e:niz}) we obtain
\begin{eqnarray*}
\limsup_{n \to \infty} \Pr \bigg(
       \sum_{i=1}^{n} \frac{X_{i}^{2}}{a_{n}^{2}}
       1_{ \big\{ \frac{|X_{i}|}{a_{n}} \leq u \big\} }  > \epsilon
       \bigg)
& \leq & \limsup_{n \to \infty} \epsilon^{-1} n \mathrm{E} \bigg( \frac{X_{1}^{2}}{a_{n}^{2}}
       1_{ \big\{ \frac{|X_{1}|}{a_{n}} \leq u \big\} } \bigg)\\[0.6em]
& = & \epsilon^{-1} u^{2-\alpha} \frac{\alpha}{2-\alpha},
\end{eqnarray*}
 and now letting $u \to 0$, since $2-\alpha >0$,
 $$\lim_{u \to 0} \limsup_{n \to \infty} \Pr \bigg(
       \sum_{i=1}^{n} \frac{X_{i}^{2}}{a_{n}^{2}}
       1_{ \big\{ \frac{|X_{i}|}{a_{n}} \leq u \big\} }  > \epsilon
       \bigg)=0.$$
Therefore we conclude that
 $$ \lim_{u \to 0} \limsup_{n \to \infty} \Pr \bigg(
     \sup_{t \in [0,1]} \|L_{n}(t) - \widetilde{L}_{n}^{(u)}(t)\| >  \epsilon \bigg) = 0.$$
\end{proof}

\begin{rem}\label{r:jointdepend}
From the proof of Theorem~\ref{t:functconvergence2} it follows that the components of the limiting process $L=(V, W)$ can be expressed as functionals of the limiting point process $N = \sum_{i} \sum_{j} \delta_{(T_{i}, P_{i}\eta_{ij})}$ from relation (\ref{e:BaTa}), i.e.
$$ V(\,\cdot\,) =  \left\{ \begin{array}{lc}
                                  \displaystyle \sum_{T_{i} \leq\,\cdot} \sum_{j}P_{i}\eta_{ij} - (\,\cdot\,) (p-q)\frac{\alpha}{1-\alpha}, & \quad \alpha \in (0,1),\\[1.2em]
                                  \displaystyle \lim_{u \to 0} \bigg( \sum_{T_{i} \le\,\cdot} \sum_{j}P_{i}\eta_{ij} 1_{\{ P_{i}|\eta_{ij}| > u \}} - (\,\cdot\,)  \int_{u < |x| \leq 1} x\mu(\rmd x) \bigg), & \quad \alpha \in [1,2).
                                 \end{array}\right.$$
where the limit in the latter case holds almost surely uniformly on $[0,1]$ (along some subsequence), and
$$ W(\,\cdot\,) = \sum_{T_{i} \leq\,\cdot} \sum_{j}P_{i}^{2}\eta_{ij}^{2}.$$
 If $\alpha \in (0,1)$, the centering function in the definition of the stochastic process $L_{n}$
 can be removed (for more details see Section 3.5.2 in Basrak et al.~\cite{BKS}).
\end{rem}

\section{Self-normalized partial sum processes}
\label{S:Selfnormalized}

The function $\pi \colon D([0,1], \mathbb{R}^{2}) \to \mathbb{R}^{2}$, definied by $\pi(x)=x(1)$, is continuous with respect to the weak $M_{1}$ topology on $D([0,1], \mathbb{R}^{2})$ (see Theorem 12.5.2 in Whitt~\cite{Whitt02}), and therefore under the conditions from Theorem~\ref{t:functconvergence2}, the continuous mapping theorem yields $\pi(L_{n}) \dto \pi(L)$ as $n \to \infty$, that is
$$  \bigg( \sum_{i = 1}^{n} \frac{X_{i}-c_{n}}{a_{n}},  \sum_{i = 1}^{n} \frac{X_{i}^{2}}{a_{n}^{2}} \bigg)
    \dto  (V(1), W(1)),$$
    with the limiting process as given in Remark~\ref{r:jointdepend}
and
\begin{equation*}\label{e:centbnew}
 c_{n} = a_{n}b_{n} = \mathrm{E} \big( X_{1} 1_{ \{ |X_{1}| \leq a_{n} \} } \big).
\end{equation*}
In particular
$$ \frac{\zeta_{n}^{2}}{a_{n}^{2}} = \sum_{i = 1}^{n} \frac{X_{i}^{2}}{a_{n}^{2}}  \dto W(1) = \sum_{i} \sum_{j}P_{i}^{2}\eta_{ij}^{2},$$
where the limiting random variable has a stable distribution.

\begin{thm}\label{t:functSN}
Let $(X_{n})$ be a sequence of random variables which satisfies all conditions in Theorem~\ref{t:functconvergence2}.
 Then
$$ \frac{S_{\lfloor n\,\cdot \rfloor}- \floor{n\,\cdot}c_{n}}{\zeta_{n}} \dto \frac{V(\,\cdot\,)}{\sqrt{W(1)}} \qquad \textrm{as} \ n \to \infty,$$
in $D([0,1], \mathbb{R})$ endowed with the $M_{1}$ topology, whit $L = (V,W)$ as described in Theorem~\ref{t:functconvergence2}.
\end{thm}
\begin{proof}
By Theorem~\ref{t:functconvergence2},
$L_{n} \dto L$ in $D([0,1], \mathbb{R}^{2})$
with the weak $M_{1}$ topology. From this convergence relation, since
$$\Pr[L \in D([0,1], \mathbb{R}) \times D_{m}([0,1], \mathbb{R})]=1,$$
by Lemma~\ref{l:M1contconst} and the continuous mapping theorem
we obtain, as $n \to \infty$,
\begin{equation}\label{e:mainconvSN}
      \widetilde{L}_{n}(\,\cdot\,) := \bigg( \sum_{i = 1}^{\lfloor n \, \cdot \, \rfloor} \frac{X_{i} - c_{n}}{a_{n}} ,  \sum_{i = 1}^{n} \frac{X_{i}^{2}}{a_{n}^{2}} \bigg)
    \dto \widetilde{L}(\,\cdot\,) :=  (V(\,\cdot\,), W(1))
\end{equation}
in $D([0,1], \mathbb{R}) \times D_{m}([0,1], \mathbb{R})$ with the weak $M_{1}$ topology.
Define the function $h \colon D([0,1], \mathbb{R}) \times C^{\uparrow}_{0}([0,1], \mathbb{R}) \to D([0,1], \mathbb{R})$ by
$ h(x,y) = x/\sqrt{y}$. A version of Lemma~\ref{l:M1div} (with the square root in the denominator) implies that
$h$ is a continuous function
when $D([0,1], \mathbb{R}) \times C^{\uparrow}_{0}([0,1], \mathbb{R})$ is endowed with the weak $M_{1}$ topology and $D([0,1], \mathbb{R})$ is endowed with the standard $M_{1}$ topology. Since
$$ \Pr[\widetilde{L} \in D([0,1], \mathbb{R}) \times C^{\uparrow}_{0}([0,1], \mathbb{R})]=1,$$
 an application of the continuous mapping theorem yields $h(\widetilde{L}_{n}) \dto h(\widetilde{L})$ as $n \to \infty$, that is
 $$ \frac{1}{\zeta_{n}}  \sum_{i = 1}^{\lfloor n \, \cdot \, \rfloor} (X_{i} - c_{n}) \dto  \frac{V(\,\cdot\,)}{\sqrt{W(1)}}$$
  with the $M_{1}$ topology.
\end{proof}

\section*{Acknowledgment}
This work was supported by University of Rijeka research grant uniri-iskusni-prirod-23-98.


\end{document}